\documentclass[10pt]{article}
\usepackage[utf8]{inputenc}
\usepackage[margin=1in]{geometry}

\usepackage[english]{babel}
\usepackage{graphicx,epstopdf,epsfig}
\usepackage{amsfonts,epsfig,fancyhdr,graphics, hyperref,amsmath,amssymb}
\usepackage{amsmath}
\usepackage{amsfonts}
\usepackage{amssymb}
\usepackage{amsthm}
\usepackage{systeme}
\usepackage{xcolor}
\usepackage{hyperref}
\usepackage{ulem}

\title{On Some Quaternionic Hadamard Matrices of Small Order}

\newcommand{\reals}{\mathbb{R}}
\newcommand{\complex}{\mathbb{C}}
\newcommand{\naturals}{\mathbb{N}}
\newcommand{\quaternions}{\mathbb{H}}
 
\newtheorem{theorem}{Theorem}
\newtheorem{lemma}{Lemma}
\newtheorem{example}{Example}
\newtheorem{definition}{Definition}
\newtheorem{proposition}{Proposition}

\newtheorem{question}{Question}
\newtheorem{corollary}{Corollary}
\newtheorem{remark}{Remark}

\begin{document}

\bibliographystyle{plain}

\setcounter{page}{1}

\thispagestyle{empty}


\author{
Logan M.\ Higginbotham\thanks{Department of Mathematical Sciences,
Campbell University, Buies Creek, North Carolina 27506, USA
(lhigginbotham@campbell.edu).} 
\and
Chase T.\ Worley\thanks{Department of Mathematics and Computing, Lander University, Greenwood, South Carolina 29649, USA (cworley@lander.edu).}
}


\maketitle

\begin{abstract}
We introduce Hadamard matrices whose entries are quaternionic. We then go on to provide classification of quaternionic Hadamard matrices of circulant core of orders 2 through 5. We also introduce quaternionic Hadamard matrices of Butson type and ways to create quaternionic Hadamard matrices from real and complex Hadamard matrices. Examples are shown that showcase how Hadamard matrices over the quaternions are richer than Hadamard matrices over the complex numbers.
\end{abstract}




\section{Introduction}

A Hadamard matrix is a square $n\times n$ matrix whose entries have absolute value 1, and whose rows (and hence columns) are pairwise orthogonal.  In the real case, we have that the entries of the matrix are $\pm 1$, and we may extend this to the complex numbers by having the entries, $h_{ij}$, of the Hadamard matrix, $H$, to have absolute value 1, i.e. $h_{ij} = e^{\theta \imath}$ for some $\theta\in\reals$.  It is easy to show that if $H$ is a real Hadamard matrix of order $n$, then $n=1,2$ or $n$ must be divisible by $4$.  This follows from the pairwise orthogonality of at least three rows. (We know that every pair of rows must differ on exactly half of its entries.)  It is still an open question if there is a Hadamard matrix of order $n=4k$ for every $k\in \naturals$.

Where there are real Hadamard matrices of certain orders, there is a complex Hadamard matrix of every order $n\in\naturals$, namely the complex Fourier matrix where the $rs^\text{th}$ entry is given by $e^{\frac{2\pi \imath (rs)}{n}}$ for $0\leq r,s\leq n-1$.  A complete classification of complex Hadamard matrices is unknown, but it is known for $n\leq 5$. We have up to a notion of equivalence that there is one Hadamard matrix of order $n=1$, one matrix of order $n=2$, one matrix of order $n=3$, infinitely many of order $n=4$, and one of order $n=5$.  It is easy to check for $n=1,2,3,4$.  For $n=5$, one can see the work of Haagerup \cite{Haagerup5by5}.  A partial classification of $n=6$ can be seen by the work of Beauchamp and Nicoara in \cite{NiBe}.  The reader should also see \cite{goodthesis} to find a catalog of the results.  Since we find that there are infinitely many complex Hadamard matrices of order 4, we begin to ask when does a particular Hadamard matrix belong to a family? On the opposite side of this coin, we can also ask when is a particular matrix \textit{not} part of a family of Hadamard matrices? Petrescu showed in \cite{Pet} and it was later confirmed and generalized by Nicoara in \cite{Ni1} that the complex Fourier matrix is not part of a family of complex Hadamard matrices when the order is prime.  A discussion of finding families of Hadamard matrices can be found in a few sources some of which include \cite{goodthesis}, \cite{NiWh}, \cite{Pet}, and \cite{WorleyThesis}.  This list is by far not exhaustive and much research has been done in this area.

After the real and complex case, one might be interested in what happens when the entries are quaternions.  Though much effort has been put into studying the real and complex cases, not much has been done in the quaternionic case.  In \cite{4x4quaternionichadamard}, the authors discuss the classification through order 4.  They find two families of quaternionic Hadamard matrices of order 4.  At the end of their paper, some open problems were stated.  One of these problems dealt with the classification of quaternionic Hadamard matrices of order 5.  In this paper, we discuss at least a partial classification of order 5 Hadamard matrices.  In particular, we find a one-parameter family of quaternionic Hadamard matrices of order 5 with \textit{circulant core}.  We begin the classification with circulant core since Haagerup showed that in the complex case all complex Hadamard matrices of order 5 are equivalent to a circulant core Hadamard matrix \cite{Haagerup5by5}.  However, unlike the complex case, we also give a one-parameter family of quaternionic Hadamard matrices of order 5 where the core is \textit{not} circulant.  In fact, we find matrices that have \text{real} entries which is impossible in the complex case.  This shows that losing the commutativity of the entries in the matrices allows for extremely interesting properties to occur.

At the end of this paper, we also discuss finding quaternionic Fourier matrices based on quaternionic solutions to the polynomial equation $x^2+1=0$ (which has infinitely many solutions over $\quaternions$).  We also discuss quaternionic Hadamard matrices when the entries solve the polynomial equation $x^k-1=0$. These matrices are commonly refered to as Butson type Hadamard matrices as first introduced in \cite{Butson} in the complex setting.

\section{Basics of Quaternions}
\begin{definition}
A \textbf{quaternion} is a number of the form $q=q_0+q_1\imath+q_2\jmath+q_3\kappa$ where $q_0,...,q_3\in\reals$ and $\imath ,~\jmath ,$ and $\kappa$ satisfy $\imath\jmath\kappa=1$ and $\imath^2=\jmath^2=\kappa^2=-1$. The collection of all quaternions is denoted $\quaternions$. We call $q_0$ the \textbf{real part} of $q$, denoted $\Re(q)$ and $q_1\imath +q_2\jmath +q_3\kappa$ the \textbf{imaginary part} of $q$, denoted $\Im(q)$. A quaternion that has a zero real part is called a \textbf{pure quaternion}.
We also call $q_1$, $q_2$, and $q_3$ the $\imath$, $\jmath$, and $\kappa$ parts of $q$ respectively. The \textbf{conjugate} of $q$, denoted $\overline{q}$ is the quaternion $\overline{q}=q_0-q_1\imath-q_2\jmath-q_3\kappa$ and the \textbf{norm} of $q$, denoted $|q|$ is $|q|=q_0^2+q_1^2+q_2^2+q_3^2$. A quaternion is \textbf{unital} if $|q|=1$ i.e. if $\overline{u}=u^{-1}$, where $u^{-1}$ is the multiplicative inverse of $u$.  
\end{definition}

Multiplication and addition of quaternions is defined in the usual way as complex numbers. However, unlike $\complex$, multiplication is \textbf{not} commutative. It's not even anti-commutative. Take $(2+i)(3-j)$ vs $(3-j)(2+i)$. They differ in their $\kappa$ parts alone. The definitions above satisfy a variety of elementary properties which are listed below. The following list is an abbreviated list from \cite{quaternionmatrixtheory}:

\begin{theorem}
Let $x$, $y$, and $z$ be quaternions. Then:
\begin{description}
    \item{1.)} $x\overline{x}=\overline{x}x$ for all $x\in\quaternions$.
    \item{2.)} $|\cdot|$ is a norm on $\quaternions$ i.e. $|x|=0$ if and only if $x=0$, $|x+y|\leq |x|+|y|$, and $|xy|=|x|~|y|$.
    \item{3.)} One can define an inner product on $\quaternions$ via $x\cdot y=x\overline{y}$
    \item{4.)} $\overline{x+y}=\overline{x}+\overline{y}$ and $\overline{xy}=\overline{y}~\overline{x}$.
    \item{5.)} $(xy)z=x(yz)$.
    \item{6.)} $\overline{x}=x$ if and only if $x\in\reals$.
    \item{7.)} $xq=qx$ for every quaternion $q$ if and only if $x\in\reals$.
    \item{8.)} If $x\neq 0$, then the multiplicative inverse of $x$, $x^{-1}$ is given by $x^{-1}=\frac{\overline{x}}{|x|^2}$. In particular, $|x|=|x^{-1}|$.
    \item{9.)} Every quaternion $q$ can be uniquely expressed as $q=c_1\imath+c_2\jmath$ where $c_1,c_2\in\complex$.
    \item{10.)} For any quaternion $x$, $x^2=|\Re(x)|^2-|\Im(x)|^2+2\Re(x)\Im(x)$.
    \item{11.)} The quaternionic equation $t^2=-1$ has infinitely many solutions.
\end{description}
\end{theorem}

Of particular note is the last item. Although there are infinitely many solutions to this equation, these solutions lie on a sphere. Note that by the previous item if $q^2=-1$, then $\Re(q)=0$ (lest $q^2$ have a non-zero imaginary part) and $|q|=1$. It then follows that $q$ is a solution to $t^2=-1$ if and only if $\Re(q)=0$ and $|q|=1$. If $q=q_1\imath +q_2\jmath+q_3\kappa$, then $q_1^2+q_2^2+q_3^2=1$. We can view $q$ as a point on a unit sphere  centered at $\Re(q)=0$. Thus, all solutions of $t^2=-1$ ``lie" on a unit sphere in $\quaternions$. It's at times useful to view $\Im(q)$ as a point on a sphere centered at $\Re(q)$ in $\quaternions$. In general, pure quaternions can be imagined to exist on a sphere in $\reals^3$ centered at the origin. This will not be the last time we observe a connection between quaternions and spheres!

At this juncture, we're in a position to define quaternionic Hadamard matrices.

\begin{definition}
Let $H$ be an $n\times n$ matrix with entries in $\quaternions$. We say $H$ is a \textbf{quaternionic Hadamard matrix of order n} provided every entry if $H$ is $1$ and any inner product of two rows or two columns is zero.
\end{definition}

Since $\quaternions$ is non-commutative, one can conjugate by a quaternion $u$ i.e. $q\to uqu^{-1}$. To avoid confusion, we the term group conjugate as opposed to conjugate to distinguish this action from the notion of conjugation above. It turns out we can use group conjugation on a quaternionic Hadamard matrix to create another quaternionic Hadamard matrix.

\begin{definition}
To \textbf{group conjugate} $q$ by $u$ is to perform the action $uqu^{-1}$. Two quaternions $q$ and $r$ are \textbf{group conjugal} if one can find a $u\in\quaternions$ so that $q=uru^{-1}$. 
\end{definition}

There are a variety of properties of group conjugation that are summarized below. Citation is given where appropriate:

\begin{theorem}
Let $q$ and $r$ be quaternions in $\quaternions$ and $u$ be a unital quaternion. Then:
\begin{description}
    \item{1.)} $q\sim r$ if and only if $q$ and $r$ are conjugal is an equivalence relation.
    \item{2.)} Group conjugation fixes real numbers.
    \item{3.)} Conjugation preserves norm i.e. $|q|=|uq\overline{u}|$; $q\overline{r}=0$ if and only if $uq\overline{u}\overline{ur\overline{u}}=0$.
    \item{4.)} If $\sim'$ is an equivalence relation defined via $q\sim'r$ if and only if there exists a unital quaternion $u$ so that $q=uru$ and $\sim$ is the equivalence relation given in item $1$, then $\sim$ and $\sim'$ are the same equivalence relation.
    \item{5.)} (From Prop 2.4.18 page 29 of \cite{quaternionalgebras}) The unit quaternions act on the collection of pure quaternions by conjugation. If $q$ is a pure quaternion and $u=u_0+u_1\imath+u_2\jmath+u_3\kappa$, $uq\overline{u}$ is a rotation by angle $2u_0$ about the axis determined by the line through the origin and $(u_1,u_2,u_3)$.
    \item{6.)} (From \cite{quaternionmatrixtheory}) If $q=q_0+q_1\imath+q_2\jmath+q_3\kappa$, then $q$ is group conjugal to $q_0+(\sqrt{q_1^2+q_2^2+q_3^2})\imath$.
    \item{7.)} (From \cite{Brenner}) Quaternions $q$ and $r$ are group conjugal if and only if $\Re(q)=\Re(r)$ and $|\Im(q)|=|\Im(r)|$.
\end{description}
\end{theorem}

There are a variety of consequences from this theorem. Item three tells us that if $H$ is a quaternionic Hadamard matrix, then $uH\overline{u}$ is another quaternionic Hadamard matrix.

Item six of the above theorem tells us that if $H$ is a quaternionic Hadamard matrix with an entry $x$, then we can find a unital $u$ so that $uH\overline{u}$ is a Hadamard matrix with $x$ complex.

It's also worth noting that a consequence of the last item of the above theorem is that every equivalence class $[q]$ for $\sim$ (or $\sim'$) are ``spheres" of radius $|q|$ that have the same real part as $q$ i.e. if $r\in [q]$, then $\Im(r)$ and $\Im(q)$ lie on a sphere of radius $|q|-q_0^2$ so we think of $[q]$ as a sphere in $\quaternions$ centered at $\Re(q)=\Re(r)$. Further, we can think of ``spheres of Hadamard matrices" $uH\overline{u}$ where $u$ is a unital quaternion.

With Hadamard matrices, we often dephase and permute rows and columns. Such actions remain an equivalence relation even if $H$ is a quaternionic Hadamard matrix.

\begin{proposition}
Let $H_1$ and $H_2$ be quaternionic Hadamard matrix. Define $\sim$ to be the $H_1\sim H_2$ if and only if there are permutation matrices $P_1$ and $P_2$ and unitary diagonal matrices $D_1$ and $D_2$ so that $H_1=P_1D_1H_2D_2P_2$. Then $\sim$ is an equivalence relation.
\end{proposition}

\begin{definition}
We say that a matrix $H$ is \textbf{dephased} if its first row and column contains only $1$'s.  The core of a matrix is the submatrix formed by deleting the first row and column of the original matrix.
\end{definition}

There are two equivalence relations at work in quaternionic Hadamard matrices. One is the typical dephasing/permutation relation above while the other is group conjugation. These relations in tandem can yield many quaternionic Hadamard matrices.

\begin{definition}
Two quaternionic Hadamard matrices $H_1$ and $H_2$ are \textbf{equivalent} if they are group conjugal or if they are dephasing/permutation equivalent.
\end{definition}

\begin{remark}
Group conjugation by some $0\ne q\in \quaternions$ takes a dephased matrix to a dephased matrix. This happens since $\reals$ is the center of $\quaternions$. 
\end{remark}

We end this section with a lemma that will be used in our classification of quaternionic Hadamard matrices of order five with circulant core:

\begin{lemma}
Let $q=q_0+q_1\imath+q_2\jmath+q_3\kappa$ and $r=r_0+r_1\imath+r_2\jmath+r_3\kappa$ be quaternions so that $q_0=r_0$ and $|q|=|r|$. Then there exists a unital quaternion $u$ so that $uq\overline{u}$ is complex with non-negative $\imath$ part and $ur\overline{u}$ has no $\kappa$ part.
\end{lemma}

\begin{proof}
This proof will proceed in two steps. We will first rotate by $u_1$ so that $u_1q\overline{u}_1$ and $u_1r\overline{u}_1$ have no $\kappa$ part. If both $q$ and $r$ have no $\kappa$ parts, then there is nothing to be done at this step. So assume at least one has a non-zero $\kappa$ part. Since $q_0=r_0$ and $|q|=|r|$, one can imagine $\Im(q)$ and $\Im(r)$ as points in $\reals^3$ with $\imath$, $\jmath$, and $\kappa$ as your typical $x$, $y$, and $z$ axes respectively. Further, these points lie on a sphere of radius $|q|=q_0^2$. Define the equator to be the great circle of this sphere that lies wholly in the $\imath-\jmath$ plane. $\Im(q)$ and $\Im(r)$ has at least one great circle that goes through these points and will intersect the equator in at exactly two antipodal points of the sphere as this great circle will be distinct from the equator.

The line through these antipodal points will be the axis that $u_1$ will rotate along when used as a conjugator; the angle to rotate by will be the angle between the $\imath-\jmath$ plane and the plane that cuts through the great circle determined by the points. Since conjugation by a unital quaternion determines a rotation by an axis by some angular measure, one can find a $u_1$ unital so that $u_1q\overline{u}_1$ and $u_1r\overline{u}_1$ have no $\kappa$ part.

If $q$ has no $\imath$ part, we may skip this step. So assume that $q$ has a non-zero $\imath$ part. We may now conjugate by a unital $u_2$ that will rotate along the $\kappa$ axis by some angular amount so that $\Im(u_2u_1q\overline{u}_1\overline{u}_2)$ lies wholly on the positive $\imath$ axis. This rotation will leave the $\imath-\jmath$ plane invariant, so $u_2u_1q\overline{u}_1\overline{u}_2$ will be complex with positive $\imath$ part while $u_2u_1r\overline{u}_1\overline{u}_2$ will still have no $\kappa$ part. Letting $u=u_2u_1$, the proof is finished. 
\end{proof}

\section{Quaternionic Hadamard Matrices of small order}

We will now make progress in classifying quaternionic Hadamard matrices of small order.  The goal in this section is classify up to order 4 and classify order 5 with circulant core.  Order 4 is classified in \cite{4x4quaternionichadamard}.  For sake of completeness, we will discuss in detail construction of orders 2 and 3.  
    \subsection{Order 2}
    
    Let $H$ be a quaternionic Hadamard matrix of order 2.  Then it follows that $H$ must be equivalent (by row/column permutations and row/column multiplications by unit quaternions) to the Fourier matrix, $$\begin{bmatrix} 1&1\\1&-1\end{bmatrix}.$$
    
    This can be shown equivalences $$\begin{bmatrix}
        a & b \\ c & d
    \end{bmatrix} \sim \begin{bmatrix}
        1 & 1\\ \frac{c}{a} & \frac{d}{a}
    \end{bmatrix} \sim \begin{bmatrix}
        1 & 1\\ 1 & \frac{da}{bc}
    \end{bmatrix}. $$  We want our matrices to be Hadamard.  This forces $\frac{ad}{bc} = -1$.  
    
    Since our matrix must be equivalent to a real-valued matrix, group conjugating by $0\ne q\in\quaternions$ yields the same matrix.

    \subsection{Order 3}

Recall that we may assume that every Hadamard matrix is equivalent to a dephased Hadamard matrix.
    
    Let $H$ be a $3\times 3$ \textit{dephased} Hadamard matrix.  Hence, $H$ has the form $$H= \begin{bmatrix} 1 & 1 & 1\\ 1 & a & b \\ 1 & c & d \end{bmatrix}.$$ From the orthogonality of the first and second rows and the first and second columns that 
\begin{align*}
    1+a+b &= 0\\
    1+a+c &= 0
\end{align*} which implies that $$b=-(1+a)=c.$$  Thus, $$ H= \begin{bmatrix} 1 & 1 & 1\\ 1 & a & b \\ 1 & b & d \end{bmatrix}.$$  Again, using the orthogonality of the rows and columns, it is easy to show that $a=d$.  Therefore, we have the following:
\begin{theorem}
Every $3\times 3$ Hadamard matrix is equivalent to a \textit{circulant core} Hadamard matrix.
\end{theorem}
\begin{proof}
    The proof follows from the previous discussion.
\end{proof}
Notice that the result is true in both the complex and quaternionic case.  There are no order 3 real Hadamard matrices as $3\not\equiv 0\mod 4$. 

For the rest of this section, we will assume that $$H= \begin{bmatrix} 1 & 1 & 1\\ 1 & a & b \\ 1 & b & a \end{bmatrix}$$ where $a,b\in\quaternions$, $|a|=|b|=1$, and $b=-(1+a)$.

\begin{lemma}For $a,b\in\quaternions$, $|a|=|b|=1$,
$b=-(1+a)$ implies that $$a\overline{b}=b=\overline{b}a.$$  Moreover, $\Re(b) = -1-\Re(a)$ and $\Im(b) = -\Im(a)$.
\end{lemma}

\begin{proof}
    Result follows from direct calculations.
\end{proof}

 Due to the orthogonality of the second and third rows, we have that $$1+a\overline{b}+b\overline{a}=0.$$
 Since $b=-1-a$, $\overline{b}=-1-\overline{a}$. Plugging this in for $b$ and $\overline{b}$ we have 
 $$1+a(-1-\overline{a})+(-1-a)\overline{a}=0.$$
 Distributing and using $a\overline{a}=1$, we have $a+\overline{a}=-1$. 
Since $\Re(a)=\frac{1}{2}(a+\overline{a})$ we have that $\Re(a)=-\frac{1}{2}$. 
 Since $b=-1-a$, $\Re(b)=-1-\Re(a)$ so $\Re(b)=-\frac{1}{2}$, too.
Hence, for $a_2, a_3,a_4\in\reals$, we have that $$a = -\frac12 + a_2\imath+a_3\jmath+a_4\kappa$$ where $$ \frac14+a_2^2+a_3^2+a_4^2=1.$$  This is the equation of a sphere of radius $\frac{\sqrt{3}}2$ which we can parameterize as 
\begin{align*}
    a_2 &= \frac{\sqrt{3}}2 \cos(\theta)\sin(\phi)\\
    a_3 &= \frac{\sqrt{3}}2 \sin(\theta)\sin(\phi)\\
    a_4 &= \frac{\sqrt{3}}2 \cos(\phi)
\end{align*}
for $\theta, \phi\in \reals$.

This tells us that $$a = a(\theta, \phi) = -\frac12 + \frac{\sqrt{3}}2 \cos(\theta)\sin(\phi)\imath + \frac{\sqrt{3}}2 \sin(\theta)\sin(\phi)\jmath + \frac{\sqrt{3}}2 \cos(\phi)\kappa$$ and $$b=b(\theta,\phi) = -\frac12 - \frac{\sqrt{3}}2 \cos(\theta)\sin(\phi)\imath - \frac{\sqrt{3}}2 \sin(\theta)\sin(\phi)\jmath - \frac{\sqrt{3}}2 \cos(\phi)\kappa = \overline{a}(\theta, \phi).$$

\begin{theorem}
All $3\times 3$ quaternionic Hadamard matrices belong to the family $$H(\theta, \phi) = \begin{bmatrix} 1 & 1 & 1 \\ 
1 & a(\theta, \phi)&\overline{a}(\theta,\phi) \\
1 & \overline{a}(\theta, \phi) &a(\theta,\phi) \end{bmatrix}.$$
\end{theorem}

\begin{proof}The proof follows from the above discussion.
\end{proof}

Notice that this family has the same form as the complex Hadamard matrix of order 3.  This makes sense since a complex Hadamard matrix is also a quaternionic Hadamard matrix.  What we have shown is that there are no other families of quaternionic Hadamard matrices of order three.  In fact, we have the following note:

\begin{remark}\label{equivFam}
If a family of quaternionic Hadamard matrices contain a complex Hadamard matrix, then every matrix in the family must be equivalent to a matrix that has the same form as in the complex case. 
\end{remark}

\begin{example}
Notice that $$H\Big(0, \frac\pi2\Big) = \begin{bmatrix} 1&1&1\\ 1 & -\frac12+\frac{\sqrt3}{2}\imath & -\frac12 - \frac{\sqrt3}2\imath\\ 1 & -\frac12 - \frac{\sqrt3}2\imath & -\frac12+\frac{\sqrt3}{2}\imath \end{bmatrix}$$ is the (complex-valued) Fourier matrix of order 3.
\end{example}
This example is particularly interesting since the Fourier matrix of prime orders are known not to belong to a parametric family in the complex setting \cite{Ni1},\cite{Pet}.  Once we move to the quaternionic setting, we immediately find a family containing the complex Fourier matrix (for at least the $p=3$ case) showing that it is no longer isolated. In fact, we have one ``sphere" of matrices. With this idea in mind, we could have classified the $3\times 3$ quaternionic Hadamard matrices more quickly. Let's reconsider the reduction of $3\times 3$ to circulant core:

$$\begin{bmatrix} 1 & 1 & 1\\ 1 & a & b \\ 1 & b & a \end{bmatrix}$$ where $a,b\in\quaternions$, $|a|=|b|=1$, and $b=-(1+a)$

Without loss of generality, we can group conjugate (traverse along a ``sphere" of Hadamard matrices) so that $a\in\complex$. We would still have $b=-(1+a)$ and since $a$ is complex, we now have $b$ is complex too. It's known that the only complex Hadamard matrix of order $3$ is the Fourier matrix, so we have that the only $3\times 3$ matrix up to equivalence (both in the sense of dephasing and row/comlumn permutation and group conjugation) is the $3\times 3$ Fourier matrix or its conjugate. We ultimately have only one matrix sphere of solutions whose representative we can choose to be the complex $3\times 3$ Fourier matrix or its conjugate. Group conjugation is a powerful tool that will be utilized in the higher orders in this paper.
  
    \subsection{Order 4}
    
    In \cite{4x4quaternionichadamard}, a complete derivation of the order 4 is discussed.  There are two families of quaternionic Hadamard matrices of order 4. One of the families discussed is the following example.
    
    \begin{example}
        Consider the generic family from \cite{4x4quaternionichadamard}. For $a=a_1+a_2\imath+a_3\jmath$ with $\|a\|=1$ and $a\ne -1$, $x\in span\{\imath, \jmath\}$ with $\|x\|=1$.  For $\hat{a} = a_1+a_2\imath$ and $$b=\left(\frac{1+\hat{a}}{|1+\hat a|}\imath\right)^2, c = \left(x\frac{1+a}{|1+a|}\right)^2, d =\left(x\frac{1+\hat{a}}{|1+\hat a|}\imath\right)^2,  $$ the matrix $$\begin{bmatrix}1&1&1&1 \\ 1&a&b&-1-a-b\\ 1&c&d&-1-c-d\\ 1 & -1-a-c & -1-b-d & 1+a+b+c+d \end{bmatrix} $$ is an order 4 quaternionic Hadamard matrix.
    \end{example}
    
    No one has classified all quaternionic Hadamard matrices of order four with circulant core. We will do so here:
    
    \begin{theorem}
Any quaternionic Hadamard matrix of order 4 with circulant core is equivalent to $F_2\otimes F_2$.
\end{theorem}

\begin{proof}
    Using the dephased matrix, you get the equations 
    $$\begin{cases}
    a+b+c+1=0\\
    1+a\overline{c}+b\overline{a}+c\overline{b}=0
    \end{cases}$$
    
    Without loss of generality, we may assume that $a=a_0+a_1\imath$ and let $b=b_0+b_1\imath+b_2\jmath+b_3\kappa$ and $c=c_0+c_1\imath+c_2\jmath+c_3\kappa$ if $a\not \in \complex$ we may group conjugate the matrix by an appropriate $0\ne q\in\quaternions$ to have $qaq^{-1}\in\complex$. Due to the first equation, we know that $b+c$ is complex so equating components we get $c_2=-b_2$ and $c_3=-b_3$. Using the first and second equations and equating components, we find the following system of polynomial equations over the real numbers:
    $$\begin{cases}
    a_0+b_0+c_0+1=0\\
    a_1+b_1+c_1=0\\
    a_0c_0+a_1c_1+a_0b_0+a_1b_1+c_0b_0+c_1b_1-b_2^2-b_3^2+1=0\\
    -a_0c_1+a_1c_0-b_0a_1+b_1a_0-b_1c_0+b_0c_1=0\\
    2a_0b_2-2a_1b_3-b_0b_2-b_2c_0+c_1b_3+b_1b_3=0\\
    2a_0b_3+2a_1b_2-b_3c_0-c_1b_2-b_2b_1-b_3b_0=0
    \end{cases}$$
    One can compute the reduced Groebner basis in reverse lexicographic monomial ordering to obtain the equivalent system
    $$\begin{cases}
     9b_1^2c_0^2 + 6b_1^2c_0 + 9b_1^2c_1^2 + b_1^2 + 9b_1c_0^2c_1 + 6b_1c_0c_1 + 9b_1c_1^3 + b_1c_1 + 9b_2^2c_1^2 + 9b_3^2c_1^2 \\ \hspace{2.5in}
      + 9c_0^2c_1^2 + 6c_0c_1^2 + 9c_1^4 + 7c_1^2=0\\
     2b_1 + 7c_1 + 3b_0b_1 + 9b_1c_0 + 6c_0c_1 + 9b_1c_0^2 + 9b_1c_1^2 + 9b_1^2c_1 + 9b_2^2c_1 + 9b_3^2c_1 \\ \hspace{2.5in} + 9c_0^2c_1 + 9c_1^3 + 9b_0b_1c_0=0\\
     9b_2^3 + 9b_2b_3^2 + 9b_2c_0^2 + 6b_2c_0 + 9b_2c_1^2 + 7b_2=0\\
     9b_2^2b_3 + 9b_3^3 + 9b_3c_0^2 + 6b_3c_0 + 9b_3c_1^2 + 7b_3=0\\
     b_0^2 + b_0c_0 + b_0 + b_1^2 + b_1c_1 + b_2^2 + b_3^2 + c_0^2 + c_0 + c_1^2 + 1=0\\
     2b_2 + 3b_0b_2 + 3b_2c_0=0\\
     2b_3 + 3b_0b_3 + 3b_3c_0=0\\
     c_1 - b_1 + 3b_0c_1 - 3b_1c_0=0\\
     b_1b_2 + b_2c_1=0\\
     b_1b_3 + b_3c_1=0\\
     a_0 + b_0 + c_0 + 1=0\\
     a_1 + b_1 + c_1=0

    \end{cases}$$

    Using the 9th equation, we have that $c_1=-b_1$ or $b_2=0$. If $c_1=-b_1$, then $c=c_0-b_1\imath-b_2\jmath-b_3\kappa$ so $b+c$ is real. Since $a+b+c+1=0$, this will imply that $a$ is real. We may then group conjugate to make $b$ complex (while $a$ remains real); using $a+b+c+1=0$, we'd have that $c$ is complex too hence we would have that the matrix is $F_2\otimes F_2$, where $F_2$ is the $2\times 2$ Fourier matrix. This is because it's known that the only order 4 complex Hadamard matrix with circulant core (up to equivalence) is $F_2\otimes F_2$ (see \cite{4x4complexhadamard}). 
    
    If $b_2=0$ and $c_1\neq-b_1$, then by the 10th equation we have that $b_3=0$ too hence $b$ and $c$ are complex numbers and by similar reasoning as the previous case we would have that $H$ is equivalent to $F_2\otimes F_2$.
    \end{proof}
    
    \subsection{Order 5}
    Next we begin to classify all quaternionic Hadamard matrices of order 5.  We begin by assuming that we have a family of quaternionic Hadamard matrices.  See Note \ref{equivFam}.  We know that there is only ONE complex Hadamard matrix of order 5, namely the Fourier matrix which is equivalent to a circulant core Hadamard matrix \cite{Haagerup5by5}.  If we find a family of quaternionic Hadamard matrices of order 5 which contains a complex Hadamard matrix, then the family must be equivalent to a family with circulant core.  Thus, we begin our attempt with looking at circulant core Hadamard matrices, but first, a note:

    \begin{remark}
        Two non-equivalent families of quaternionic Hadamard matrices of order 5 can't both contain the complex Fourier matrix of order 5.
    \end{remark}
    
     This is because if $A, B\in M_5(\complex)$, both Hadamard, then $A$ and $B$ are both equivalent to the complex Fourier matrix, $F_5$, from Haagerup \cite{Haagerup5by5}.  Thus, the two families must be equivalent to each other.
    
    \subsubsection{Circulant Core Hadamard matrices of order 5}
    
    In our search to classify quaternionic Hadamard matrices of order 5, we are going to focus on circulant core matrices.  In the complex case, we know that all complex Hadamards are equivalent to the Fourier matrix \cite{Haagerup5by5}.  Indeed, it can be shown that the Fourier matrix of order 5 is indeed equivalent to a circulant core matrix.  For a quick discussion, one can see \cite{NiWor} and \cite{WorleyThesis}.
    
    Let $H$ be a circulant core quaternionic Hadamard matrix of order 5.  Then $H$ has the form: 
    $$ H = 
    \begin{bmatrix}
    1 & 1 & 1 & 1 & 1\\
    1 & a & b & c & d\\
    1 & d & a & b & c\\
    1 & c & d & a & b\\
    1 & b & c & d & a
\end{bmatrix}.$$

If none of the entries are complex, we can group conjugate so that (without loss of generality) $a\in\complex$. If all of the entries are complex, then $H$ is equivalent to the Fourier matrix.  Note that if three of the entries in the core are complex, then the fourth entry must also be complex since the sum of $1+a+b+c+d=0$.  Thus, we really have a few cases: all entries of $H$ are complex, i.e. $H$ is the complex Fourier matrix, two entries of the core are complex and two entries are in $\quaternions\setminus\complex$, or there is exactly one entry of $H$ that is complex.. If only one entry is complex that isn't $a$ we can always move row 5 to row 2 and shift the other rows down, i.e. perform a cyclic permutation of the core to move the complex number along the diagonal. Since we may assume that $a\in \complex$, we only need to let $b\in\complex$ or $c\in\complex$.  If $d\in\complex$ and $b,c\in\quaternions\setminus\complex$ then we may perform a cyclic permutation on the core to make so that the first two entries of the core are complex.  Hence, when we have two complex numbers in the core either $a,b\in\complex$ or $a,c\in \complex$. 

In conclusion, we have the following non-trivial cases: $a$ and $b$ being complex, $a$ and $c$ being complex, and only $a$ being complex.

If $a,b\in\complex$, then $H$ belongs to the sphere-family, the one-parameter family where $$a= -\frac14 \pm \sqrt{\frac{15}{16}}\ \imath$$ and $$c = -\frac14 \mp \sqrt{\frac{5}{48}}\ \imath + \sqrt{\frac56} \cos(t)\ \jmath +\sqrt{\frac56}\sin(t)\ \kappa$$ in the matrix: 
$$
\begin{bmatrix}
    1&1&1&1&1\\
    1 & a & \overline{a} & c & \overline{c}\\
    1 & \overline{c} & a & \overline{a} & c\\
    1 & c & \overline{c} & a & \overline{a} \\
    1 & \overline{a} & c & \overline{c} & a  \\
\end{bmatrix}.
$$
To find other matrices in this sphere multiply on the left by $0\ne q\in\quaternions$ and on the right by $q^{-1}$.

If $a,c\in\complex$, then $H$ is equivalent to the complex Fourier matrix of order 5.

\begin{theorem}
If a circulant core quaternionic Hadamard matrix of order 5 has two complex entries in the core (after appropriate group conjugation), then either $H$ is equivalent to the complex Fourier matrix of order 5 or $H$ belongs to the sphere-family stated previously.
\end{theorem}

The proof relies on Gr{\"o}bner bases and was inspired by discussions in appendix C \cite{goodthesis}.

\begin{proof}
    Let $$ H = 
    \begin{bmatrix}
    1 & 1 & 1 & 1 & 1\\
    1 & a & b & c & d\\
    1 & d & a & b & c\\
    1 & c & d & a & b\\
    1 & b & c & d & a
\end{bmatrix}.$$ 

Using the fact that $H$ is a Hadamard matrix, it is easy to check that the entries of $H$ must satisfy the equations: 
$$\begin{cases}
    |a| = |b| = |c| = |d| = 1,\\
    1 + a + b + c + d = 0,\\
    1+a\overline{d}+b\overline{a}+c\overline{b}+d\overline{c}=0,\\
    1+a\overline{c}+b\overline{d}+c\overline{a}+d\overline{b}=0  .
\end{cases}$$
\textbf{Case 1:}
Assume that $a,b\in\complex$.
Replacing the above system containing the quaternion variables with ``vector" forms of $a,b,c,$ and $d$, i.e.
\begin{align*}
    a &= a_0 + a_1\imath,\\
    b &= b_0 + b_1\imath,\\
    c &= c_0 + c_1\imath + c_2\jmath + c_3\kappa,\\
    d &= d_0 + d_1\imath + d_2\jmath +d_3\kappa,
\end{align*}
we get the system of polynomial equations in real variables:
$$
\begin{cases}
    a_0^2 + a_1^2 &= 1,\\
    b_0^2 + b_1^2 &= 1,\\
    c_0^2 + c_1^2 + c_2^2 + c_3^2 &= 1,\\
    d_0^2 + d_1^2 + d_2^2 + d_3^2 &= 1,\\
    1 + a_0 + b_0 + c_0 +d_0 &= 0,\\
    a_1 + b_1 + c_1 + d_1 &= 0,\\
    c_2 + d_2 &= 0,\\
    c_3 + d_3 &= 0,\\
    1 + 2 a_0 c_0 + 2 a_1 c_1 + 2 b_0 d_0 + 2b_1 d_1 &= 0,\\
\end{cases}
$$
Included in this system is another four equations stemming from the real part, and coefficients corresponding to $\imath, \jmath, $ and $\kappa$ in the equation \\ $1+a\overline{c}+b\overline{d}+c\overline{a}+d\overline{b}=0$.  We immediately find that $d_2 = -c_2$ and $d_3 = -c_3$.  So we may reduce our system.  Hence, by making the appropriate substitutions, we have that 
\begin{align*}
    a &= a_0 + a_1\imath,\\
    b &= b_0 + b_1\imath,\\
    c &= c_0 + c_1\imath + c_2\jmath + c_3\kappa,\\
    d &= d_0 + d_1\imath - c_2\jmath - c_3\kappa,
\end{align*}
and we have a system of 11 polynomial equations.  Using Mathematica, we can find a Gr{\"o}bner basis using degree reverse lexicographic ordering for our system of equations.  The basis that we find includes 38 polynomials.

Some of the equations that we find are:

$$\begin{cases}
    c_3(1+4d_0) &= 0,\\
    c_2(1+4d_0) &= 0,\\
    c_3(1+4c_0) &= 0,\\
    c_2(1+4c_0) &= 0,\\  
    c_3(c_1+d_1) &= 0,\\
    c_2(c_1+d_1) &= 0
\end{cases}$$

If $c\in\complex$ (which also implies $d\in\complex$), then we gain no new information, i.e. $H\in M_5(\complex)$, but if $c\not\in \complex$, then we know that either $c_2$ or $c_3$ is nonzero.  This yields 
$$\begin{cases}
    d_0 &= -\frac14,\\
    c_0 &= -\frac14,\\
    d_1 &= -c1,
\end{cases}$$
immediately implying that $d=\overline{c}$.

Substituting these values into the Gr{\"o}bner basis we find that 
$$\begin{cases}
    \frac12 + a_0 + b_0 &= 0,\\
    a_1+b_1 &= 0
\end{cases}$$
giving us that $b_0 = -\frac12+a_0$ and $b_1 = -a_1$. This gives us that 
\begin{align*}
    a &= a_0 + a_1\imath,\\
    b &= \Big(-\frac12-a_0\Big) - a_1\imath,\\
    c &= -\frac14 + c_1\imath + c_2\jmath + c_3\kappa,\\
    d &= -\frac14 - c_1\imath - c_2\jmath - c_3\kappa.
\end{align*}
After making the substitutions into the basis and simplifying, we find the equations 
$$
\begin{cases}
    \frac12+2a_0 &= 0,\\
    -\frac{15}{16} +a_1^2 &= 0.
\end{cases}
$$
Thus, $a_0 = -\frac14\implies b_0 =-\frac14 \implies b=\overline{a}$, and $a_1 = \pm\sqrt{\frac{15}{16}}$.  This implies that 
\begin{align*}
    a &= -\frac14 \pm \sqrt{\frac{15}{16}}\ \imath,\\
    b &= \overline{a},\\
    c &= -\frac14 + c_1\imath + c_2\jmath + c_3\kappa,\\
    d &= \overline{c}.
\end{align*}
Returning to the Gr{\"o}bner basis and after substituting the values that we have, we can also use the equations: 
$$
\begin{cases}
    c_2(-5+48c_1^2) &=0,\\
    c_3(-5+48c_1^2) &= 0.
\end{cases}
$$

Given that $c\not\in\complex$, we have that $c_1^2 = \frac{5}{48}$.  It is an easy check using the equations in the basis that $a_1$ and $c_1$ must have opposite signs.  Hence $$a_1 = \pm\sqrt{\frac{15}{16}} \implies c_1 \mp\sqrt{\frac{5}{48}}. $$ Lastly, after making the appropriate substitutions into the basis, we find that $$c_2^2+c_3^2 = \frac56,$$ i.e. $c_2 = \sqrt{\frac56}\cos(t)$ and $c_3 =\sqrt{\frac56}\sin(t)$ for some $t\in\reals$.

Therefore if $a,b\in \complex$ and $c,d\in\quaternions\setminus\complex$, then 

$$a= -\frac14 \pm \sqrt{\frac{15}{16}}\ \imath$$ and 
$$c = -\frac14 \mp \sqrt{\frac{5}{48}}\ \imath + \sqrt{\frac56} \cos(t)\ \jmath +\sqrt{\frac56}\sin(t)\ \kappa$$ 
with $b=\overline{a}$ and $d=\overline{c}$.

\textbf{Case 2:}

Assume that $a,c\in\complex$.  Using the fact that $1+a+b+c+d=0$, we find that 
\begin{align*}
    a &= a_0 + a_1 \ \imath,\\
    b &= b_0 + b_1\ \imath + b_2 \ \jmath + b_3\ \kappa,\\
    c &= c_0 + c_1 \ \imath,\\
    d &= d_0 + d_1\ \imath - b_2\ \jmath -b_3 \ \kappa.
\end{align*}
Again, using the equations from the ``vector" form of $a,b,c,$ and $d$ and the Hadamard condition, we can find a Gr{\"o}bner basis using Mathematica and degree reverse lexicographic ordering which contains 36 polynomials. In this basis, we have the equations: 
$$
\begin{cases}
b_2(3+2c_0) &= 0,\\
b_3(3+2c_0) &= 0.
\end{cases}
$$

If $b\in \complex$, then $a,b,c,d\in\complex$ and $H$ is equivalent to the complex Fourier matrix.  But if $b\not\in \complex$, then either $b_2\ne 0$ or $b_3\ne 0$.  Both would imply that $c_0 = -\frac32$.  This would contradict the fact that $c$ has norm 1.  Hence if $a,c\in\complex$, then $a,b,c,d\in\complex$.  Therefore, $H$ is equivalent to the complex Fourier matrix when $a,c\in\complex$.  This gives us a sphere of matrices where we group conjugate by a quaternion.

We also note that these families are in two distinct spheres. Consider the family with $a,b\in\complex$ and $c,d\not\in\complex$.  For some $0\ne q\in\quaternions$, then 
$$
\begin{bmatrix}
    1 & 1 & 1 & 1 & 1\\
    1 & qaq^{-1} & q\overline{a}q^{-1} & q c q^{-1} & q\overline{c}q^{-1} \\
    1 & q\overline{c}q^{-1} & qaq^{-1} & q\overline{a}q^{-1} & q c q^{-1}  \\
    1 & q c q^{-1} & q\overline{c}q^{-1} & qaq^{-1} & q\overline{a}q^{-1}   \\
    1 & q\overline{a}q^{-1} & q c q^{-1} & q\overline{c}q^{-1} & qaq^{-1}    \\
\end{bmatrix}
$$
is a family of quaternionic Hadamard matrices of order 5 with circulant core with $a,c$ defined above in the proof.  This particular family does not contain any complex Hadamard matrices.  If it did, then there is some $q\in \quaternions$ such that $qaq^{-1}, qcq^{-1}\in \complex$.  This would imply that $qaq^{-1}qcq^{-1} = qcq^{-1}qaq^{-1}$.  This would suggest that $ac=ca$, but it is an easy check to see that $ac\ne ca$.
\end{proof}

What happens if there is only one entry that is complex?  Again we may suppose that $a\in\complex$ and $b,c,d\in\quaternions\setminus\complex$. Using Mathematica to construct the Gr{\"o}bner basis with the Hadamard condition equations and
\begin{align*}
    a &= a_0 + a_1 \ \imath,\\
    b &= b_0 + b_1 \ \imath + b_2 \ \jmath + b_3\ \kappa,\\
    c &= c_0 + c_1 \ \imath + c_2 \ \jmath + c_3\ \kappa,\\
    d &= d_0 + d_1 \ \imath + d_2 \ \jmath + d_3\ \kappa,\\
\end{align*}
    we find the equations 
    $$
    \begin{cases}
        c_3d_2 - c_2d_3 &= 0,\\
        b_0c_3 - c_3d_0 &= 0,\\
        b_0c_2 - c_2d_0 &= 0.
    \end{cases}
    $$
    This implies that either $c\in\complex$ ($c_2=c_3=0$ which contradicts that $a$ is the only complex entry), or $b_0 = d_0$. 
    
    Thus, we may assume that the real parts of $b$ and $d$ are equal to each other.  Hence, 
    
    \begin{align*}
    a &= a_0 + a_1 \ \imath,\\
    b &= b_0 + b_1 \ \imath + b_2 \ \jmath + b_3\ \kappa,\\
    c &= c_0 + c_1 \ \imath + c_2 \ \jmath + c_3\ \kappa,\\
    d &= b_0 + d_1 \ \imath + d_2 \ \jmath + d_3\ \kappa.\\
\end{align*}
Using the fact that $1+a+b+c+d=0$, we have that $1+a_0+2b_0+c_0=0$.  Replacing $c_0$, $d_1, d_2,$ and $d_3$ we get that:
\begin{align*}
    a &= a_0 + a_1 \ \imath,\\
    b &= b_0 + b_1 \ \imath + b_2 \ \jmath + b_3\ \kappa,\\
    c &= (-1-a_0-2b_0) + c_1 \ \imath + c_2 \ \jmath + c_3\ \kappa,\\
    d &= b_0 + (-a_1-b_1-c_1) \ \imath + (-b_2-c_2) \ \jmath + (-b_3-c_3)\ \kappa.\\
\end{align*}
This allows us to recompute the Gr{\"o}bner basis in degree reverse lexicographic monomial ordering.  From this we get the following equations:
$$
\begin{cases}
    b_3c_2 -c_3b_2 &= 0,\\
    (1+2a_0+2b_0)(2b_3+c_3) &= 0,\\
    (1+2a_0+2b_0)(2b_2+c_2) &= 0,\\
    (1+2a_0+2b_0)(a_1+2b_1+c_1) &= 0.\\
\end{cases}
$$
If $1+2a_0+2b_0 \ne 0$, then $b_2=-2c_2$ and $b_3=-2c_3$.  Substituting these values into the Gr{\"o}bner basis we find that either $b_3 = 0$ or $b_0 = -\frac32$ which contradicts the fact that $|b|=1$.  We also get that $b_2=0$ or $b_0 = -\frac32$.  Hence, it must be the case that $b_2=0=b_3$ which implies that $b\in\complex$.  This contradicts the fact that $a$ is the only complex entry.  Hence, it must be the case that $1+2a_0+2b_0=0$.Thus, we get that $b_0 = -\frac12-a_0$ and $c_0=a_0$.  Therefore, 
\begin{align*}
    a &= a_0 + a_1 \ \imath,\\
    b &= \Big(-\frac12-a_0\Big) + b_1 \ \imath + b_2 \ \jmath + b_3\ \kappa,\\
    c &= a_0 + c_1 \ \imath + c_2 \ \jmath + c_3\ \kappa,\\
    d &= \Big(-\frac12-a_0\Big) + (-a_1-b_1-c_1) \ \imath + (-b_2-c_2) \ \jmath + (-b_3-c_3)\ \kappa.\\
\end{align*}
Since we know that $d_0=b_0=-\frac12-a_0$ (which also implies that $a_0=c_0$), we can use the lemma in Section 2 that will allow us to group conjugate the matrix by some $0\ne q\in\quaternions$ so that $qbq^{-1}\in\complex$ and $qdq^{-1}$ has no $\kappa$ part.  We may also assume that $b_1>0$.  The case when $b_1<0$ is the complex conjugate case (one can complex conjugate by group conjugating by $\jmath$; see \cite{quaternionmatrixtheory}).  Since we are conjugating the entire matrix by $q$, all of the entries have norm 1, and the real parts are left unchanged by conjugation we may replace $qaq^{-1}$ with $a$ and likewise for $b$, $c$, and $d$.  This and the fact that $1+a+b+c+d=0$ gives us that 
\begin{align*}
    a &= a_0 + a_1 \ \imath + a_2 \ \jmath + a_3 \ \kappa,\\
    b &= \Big(-\frac12-a_0\Big) + b_1 \ \imath,\\
    c &= a_0 + (-a_1-b_1-d_1) \ \imath + (-a_2-d_2) \ \jmath + -a_3\ \kappa,\\
    d &= \Big(-\frac12-a_0\Big) + d_1 \ \imath + d_2 \ \jmath.\\
\end{align*}
Since $1+a\overline{d}+b\overline{a}+c\overline{b}+d\overline{c}=0$, the $\jmath$ part of the equation implies that $a_3=0$ or $d_2=0$.  If $d_2=0$, then $d\in\complex$ which gives us that $b$ and $d$ are complex, so we have the complex Fourier matrix by the above work.  So we may assume that $a_3=0$ which implies $c_3=0$.  Thus, 
\begin{align*}
    a &= a_0 + a_1 \ \imath + a_2 \ \jmath,\\
    b &= \Big(-\frac12-a_0\Big) + b_1 \ \imath,\\
    c &= a_0 + (-a_1-b_1-d_1) \ \imath + (-a_2-d_2) \ \jmath,\\
    d &= \Big(-\frac12-a_0\Big) + d_1 \ \imath + d_2 \ \jmath.\\
\end{align*}
Using the fact that $\|b\|=1$ and $b_1>0$, we have that $$ b_1 = \sqrt{1-\Big(\frac12+a_0\Big)^2} = \frac{\sqrt{3-4a_0-4a_0^2}}{2}=\frac{\sqrt{(1-2a_0)(3+2a_0)}}{2}.$$  Using the real part of the equation $1+a\overline{d}+b\overline{a}+c\overline{b}+d\overline{c}=0$, we find that $$1-2a_0-4a_0^2-b_1^2-2b_1d_1-d_1^2-d_2^2=0$$ which implies that $$4a_0^2+2a_0-1 = (b_1+d_1)^2 +d_2^2.$$  Therefore, it must be the case that
$$0\leq 4a_0^2+2a_0-1 = \Big(2a_0+\frac12\Big)^2-\frac54.$$ Thus, it follows that $$ \frac{-1-\sqrt{5}}{4} \leq a_0 \leq \frac{-1+\sqrt{5}}{4}.$$  
Note that when $a_0 = \frac{-1-\sqrt{5}}{4}, \frac{-1+\sqrt{5}}{4}$, we can substitute into the equation to find that $a_2=c_2=d_2=0$, i.e. our matrix is complex and hence equivalent to the Fourier matrix.

Since $1=|c| = |-1-a-b-d| = |1+a+b+d|$, we have that 
\begin{align*}
    -1+a_0^2+a_1^2+a_2^2 + b_1^2+2b_1d_1+d_1^2+2a_1(b_1+d_1)+2a_2d_2+d_2^2 &= 0\\
    (b_1+d_1)^2 +2a_1(b_1+d_1)+2a_2d_2+d_2^2 &= 0
\end{align*}
Looking at the real and $\imath$ components of the equation $1+a\overline{d}+b\overline{a}+c\overline{b}+d\overline{c}=0$, we have that
\begin{align*}
    1+2a_0^2+\frac12(1+2a_0)^2+2b_1d_1-2a_1^2-2a_1(b_1+d_1)-2a_2^2-2a_2d_2 &=0\\
    1-2a_0-4a_0^2-(b_1+d_1)^2-d_2^2 &= 0
\end{align*}
which simplifies to 
\begin{align*}
     1+2a_0^2+\frac12(1+2a_0)^2+2b_1d_1-2(1-a_0^2) &= 2a_1(b_1+d_1)+2a_2d_2 \\
     1-2a_0-4a_0^2 &= (b_1+d_1)^2+d_2^2.
\end{align*}
Adding these equations together, we get exactly the relationship we got from $c$ having unit norm.  Therefore, we have that 
\begin{align*}
    0 &= 1+2a_0^2+\frac12(1+2a_0)^2+2b_1d_1-2(1-a_0^2) + 1-2a_0-4a_0^2 \\
    0 &= \frac12 + 2a_0^2+2b_1d_1 \\
    b_1d_1 &= -\frac14 (1+4a_0^2).
\end{align*}
Hence, $$d_1 = -\frac{(1+4a_0^2)}{4b_1} =\frac{-(1+4a_0^2)}{2\sqrt{3-4a_0-4a_0^2}} = \frac{-(1+4a_0^2)}{2\sqrt{(1-2a_0)(3+2a_0)}} $$ assuming $b_1\ne 0$ (which can't happen since $a_0\ne \frac12, -\frac32$).

Using the fact that $d_1^2+d_2^2=b_1^2$, we find $$d_2 = \pm \sqrt{\frac{2(1-a_0)(-1+2a_0+4a_0^2)}{-3+4a_0+4a_0^2}}.$$  Note that $d_1$ has the \textit{opposite} sign of $b_1$ as $b_1d_1<0$ from the above calculations.

Finally, we look at the $\kappa$ part of the equation  $1+a\overline{d}+b\overline{a}+c\overline{b}+d\overline{c}=0$.  We find that $$-a_2b_1 + a_2d_1 -b_1 (a_2+d_2) -a_1d_2 - (a_1+b_1+d_1)d_2 + d_1(a_2+d_2) = 0 $$ which implies $$a_2 = \frac{(a_1+b_1)}{d_1-b_1}d_2$$ assuming that $b_1\ne d_1$. Note that if $b_1=d_1$, we would have that $b=d$, an impossibility.  Recall that we find $b_1$ and $d_1$ have opposite signs.   Indeed if $b=d$, $b_0=d_0$ and $b_1=d_1$ along with $|b|=1$ forces $d_2=0$ and implies $b=d$, i.e. $d$ is complex and would fall to an earlier case. So, we may assume $d_1\neq b_1$. Hence, our choice of $a_0$ determines $b_1$, $d_1$, and $d_2$. Since the $a$ has unit norm, we have that $a_1$ and $a_2$ are on a circle of radius $\sqrt{1-a_0^2}$. This suggests that we can express $a_1$ in terms of $a_0$ and therefore $a_2$ in terms of $a_0$.  In fact, $a_1$ must satisfy the equation: $$a_0^2+a_1^2 + \Big(\frac{(a_1+b_1)d_2}{d_1-b_1}\Big)^2 -1=0. $$

This is a quadratic equation in $a_1$.  We can solve as we normally would.  This allows us to find $a_1$ in terms of just $a_0$.  Thus, we are able to eliminate the last parameter since $a_2$ will also be determined by $a_0$.  Hence, we have that there is a one-parameter family of quaternionic Hadamard matrices with circulant core containing one complex entry.  For each choice of $a_0$, we find a sphere of matrices by group conjugating by some nonzero quaternion.  

To recap our solutions, we list the components of $a,b,c,$ and $d$ in terms of $a_0$:

\begin{align*}
    b_0 &= -\frac12 -a_0\\
    b_1 &= \sqrt{1-\Big(\frac12+a_0\Big)^2} = \frac{\sqrt{3-4a_0-4a_0^2}}{2}=\frac{\sqrt{(1-2a_0)(3+2a_0)}}{2}\\
    d_1 &= \frac{-(1+4a_0^2)}{2\sqrt{(1-2a_0)(3+2a_0)}}\\
    d_2 &= \pm \sqrt{\frac{2(1-a_0)(-1+2a_0+4a_0^2)}{-3+4a_0+4a_0^2}}\\
    a_1 &= \frac{\sqrt{1-2a_0}\cdot(-1+2a_0+4a_0^2) \pm \sqrt{2(3+2a_0)(1-3a_0+2a_0^2)}}{2(a_0-1)\sqrt{3+2a_0}},\\
    a_2 &= \frac{\sqrt{2-4a_0}\cdot (\mp 1 \pm a_0) + \sqrt{(3+2a_0)(1-3a_0+2a_0^2)}}{2(1-a_0)\sqrt{1-2a_0}}.
\end{align*}

The sign of $d_2$ may be either positive or negative.  However, the signs in the formulas for $a_1$ and $a_2$ must match.  It is easy to show using Mathematica that these define a solution given that $c=-1-a-b-d$.


Now that we have classified quaternionic Hadamard matrices of order five with circulant core, there is some discussion to be had. The authors do not believe that the discussion of the family discovered above ends at its discovery. Mainly, one wonders which matrices in this family are equivalent or not. Note that if one chooses $a_0=\frac{-1-\sqrt{5}}4$ and $a_0=\frac{-1+\sqrt{5}}4$, one finds the complex Fourier matrix thus resulting in these choices of $a_0$ equivalent up to row/column permutation. Are there other such choices? Can we describe all choices of $a_0$ that result in inequivalent matrices?  Also, it is of interest to note that the previous case with two complex entries may be recovered in this case well by setting $a_0=-\frac14$.  We handled the cases separately as it is more natural to lower the number of entries that are complex rather than increase them. 

\begin{question}
    What choices in $a_0$ result in equivalent matrices?
\end{question}

It was shown in \cite{Haagerup5by5} that any complex Hadamard matrix of order five is equivalent to a Hadamard matrix of circulant core (namely the $5\times 5$ Fourier matrix). However, Haagerup's argument makes use of the commutativity of $\complex$ so his techniques can't be easily emulated in the quaternionic case. We then can ask the question:

\begin{question}
    Does there exist a quaternionic Hadamard matrix of order five that is 
\textbf{not} equivalent to a quaternionic Hadamard matrix of circulant core?
\end{question}

There is an immediate answer to this question. This shows there is a vast difference between the complex and quaternionic case.  

\subsubsection{Non-circulant core quaternionic Hadamard of order 5}

It is natural to think that all quaternionic Hadamard matrices may be equivalent to a circulant core due to the work of Haagerup in the complex case \cite{Haagerup5by5}, but we will provide an example of a non-circulant core quaternionic Hadamard matrix. The following example is \textit{impossible} in the complex case by Lemma 2.7 in \cite{Haagerup5by5}, but we find solutions over $M_5(\quaternions)$. 

\begin{example}
Consider a Hadamard matrix of the form 
$$ H=
\begin{bmatrix}
    1 & 1 & 1 & 1 & 1\\
    1 & a & b & c & d\\
    1 & b & a & d & c\\
    1 & c & d & a & b\\
    1 & d & c & b & a
\end{bmatrix}.
$$
Writing down the equations from the Hadamard condition, using the fact that the elements have unit norm, and the elements in the second row sum to 0, we can write the quaternions in vector form and proceed using a Gr{\"o}bner basis again in degree reverse lexicographic monomial ordering.  We may assume without loss of generality that $a\in \complex$.  A natural question then arises - what happens if $c$ is also a complex number? In this case, we find that \begin{align*}
    a &= \imath,\\
    b &= -\frac12\ \imath + \frac{\sqrt{3}}{2}\cos(t) \ \jmath + \frac{\sqrt{3}}{2} \sin(t)\ \kappa,\\
    c &= -1,\\
    d &= -\frac12\ \imath - \frac{\sqrt{3}}{2}\cos(t) \ \jmath - \frac{\sqrt{3}}{2} \sin(t)\ \kappa.\\
\end{align*}
It is a straight forward calculation to show that this gives us a Hadamard matrix. In fact, we get a family of matrices for the value of $t\in \reals$.  For each $t$, we also are given a sphere by conjugating the matrix by some $0\ne q\in\quaternions$.  It is also interesting to note that 
\begin{align*}
    a &= \imath,\\
     b &= \imath \cdot \Big(\cos\Big(\frac{2\pi}{3}\Big)+q\sin\Big(\frac{2\pi}{3}\Big)\Big)\\
    c &= a^2=-1,\\
    d &= \imath \cdot \Big(\cos\Big(\frac{2\pi}{3}\Big)-q\sin\Big(\frac{2\pi}{3}\Big)\Big)\\
    where q &= \sin(t)\ \jmath - \cos(t)\ \kappa
\end{align*}
and that $a^4 = b^4 = c^4 = d^4 = 1$ and $q^2=-1$. We should also note that $b,d\in [\imath]$, i.e. both $b$ and $d$ are group conjugal equivalent to $\imath$.  Both $b$ and $d$ have 0 real part and both have norm 1.  This implies that $ubu^{-1} = \imath =qdq^{-1}$ for $u,q\in\quaternions$.  This shows the fact that $b^4 = u^{-1}\imath^4u = 1$ and similarly for $d$.  We'll come back to this example in sections four and five.
\end{example}


\section{Finding quaternionic Hadamard matrices through the complex Fourier Matrix}

We wish to construct a quaternionic version of the Fourier matrix.  In other words, we wish to build a quaternionic Hadamard matrix of order $n$ with all entries $n^{\text{th}}$ roots of unity, i.e. we wish the entries of the matrix to solve the equation $x^n-1=0$.  From \cite{quaternionpolynomial}, we know that there are $n$ complex roots of unity $x_1,\cdots, x_n$.  If $x_l\in\complex\setminus\reals$, then we end up with a \text{sphere} of solutions generated by $x_l$.  This sphere is denoted by $$[x_l]=\{qx_lq^{-1}\ |\ q\in \quaternions\setminus\{0\}\}.$$  If $x_l\in\reals$, then $[x_l]=\{x_l\}$ since $\reals$ is the center of $\quaternions$.

Let $q\in \quaternions$ where $q^2=-1$, we have that $q\in [\imath]$.  Then we have that $\Re(q)=\Re(\imath) = 0$ and $|q|=|\imath|=1$.  This implies that $q$ can be found on the 3-dimensional unit sphere centered at the origin.  Hence, we may parameterize $q$ as $$q=q(\theta, \varphi)=\cos(\theta)\sin(\varphi)\imath + \sin(\theta)\sin(\varphi) \jmath + \cos(\varphi)\kappa$$ for $\theta, \varphi\in\reals$.

It can easily be shown that a quaternionic version of de Moivre's theorem holds for $q\in\quaternions$ with $q^2=-1$, i.e. $$e^{q\theta} = \cos(\theta)+q\sin(\theta).$$ 

There are some things worth noting here. First, consider the last example of the previous section; recall that $b$ and $d$ utilize $\imath$ and $q$ which are both roots of $-1$. It's quite bizarre to find $b$ and $d$ to be fourth roots of unity. Indeed, $\imath=e^{\imath(\frac{\pi}{2})}$ and $\cos(\frac{2\pi}{3})+q\sin(\frac{2\pi}{3})=e^{q(\frac{2\pi}{3})}$.  $\imath$ is a fourth root of unity while $e^{q(\frac{2\pi}{3})}$ is a third root of unity. How is their product a fourth root of unity? Take heed in that the non-commutativity of quaternions prevents us from using a common proposition in complex arithmetic:
\begin{remark}
Let $q_1$ and $q_2$ be roots of $-1$ and $\theta_1$ and $\theta_2$ be real numbers. In general, $e^{q_1\theta_1}e^{q_2\theta_2}\neq e^{q_1\theta_1+q_2\theta_2}$.
\end{remark}

Since this property fails in general, there are many more ways to concoct roots of unity over the quaternions than the complex numbers. On the other hand, group conjugation plays nice with these exponentials: 

\begin{remark}
For $e^{q\theta}$ where $\theta \in \reals$ and $q^2=-1$, group congugating by $0\ne v\in\quaternions$ leads to 
\begin{align*}
    [e^{q\theta}] \ni ve^{q\theta}v^{-1} &= v\Big(\cos(\theta)+q\sin(\theta)\Big)v^{-1}\\
    &= \cos(\theta) +vqv^{-1}\sin(\theta)\\
    &= e^{(vqv^{-1})\theta}\in e^{[q]\theta}.
\end{align*}
This shows that both spheres are subsets of each other, i.e. they are the same sphere.
\end{remark}

Next we look at our idea for this section: constructing "quaternionic" Fourier matrices of order $n$.

\begin{theorem}
There exists infinitely many dephased, non-equivalent quaternionic Hadamard matrices for all $n\geq 3$, namely $$F_n(\theta, \varphi) = \Big[e^{\frac{2\pi q(\theta, \varphi) j k}{n}}\Big]_{j,k=0}^{n-1}$$ 
where $q(\theta, \varphi)=\cos(\theta)\sin(\varphi)\imath + \sin(\theta)\sin(\varphi) \jmath + \cos(\varphi)\kappa $. Moreover, this family contains the complex Fourier matrix of order $n$ for each $n\in\naturals$.
\end{theorem}

The proof is straightforward, but the key is that each element still has absolute value 1 and the resulting sums are the same sums found in the complex case just group conjugated by $q$.  Hence, all the sums are 0 as desired.

Essentially, what we have found is a \textit{sphere} of quaternionic Hadamard matrices generated by the complex Fourier matrix of order $n$.

This is true for every $n$.  Hence, as long as $F_n\not\in M_n(\reals)$, i.e. $n\geq 3$, we get infinitely many matrices in the sphere, i.e. we have a family of quaternionic Hadamard matrices that contains the complex Fourier matrix for every $n$.  

This is an important fact since Petrescu, Nicoara, and White showed in \cite{Ni1} , \cite{Pet}, and \cite{NiWh} that the complex Fourier matrix of order $n$ is isolated amongst complex Hadamard matrices of order $n$ if and only if $n$ is prime.

\begin{example}
Let $n=4$.  Consider $w^n=e^{\frac{2\pi q n}{4}}=e^{\frac{\pi q n}{2}}=\cos(\frac{\pi}{2}n)+q\sin(\frac{\pi}{2}n)$ for $q=q(\theta, \varphi) = \cos(\theta)\sin(\varphi)\imath + \sin(\theta)\sin(\varphi) \jmath + \cos(\varphi)\kappa$. 
\begin{align*}
    n=0 &\implies w^0 = 1\\
    n=1 &\implies w^1 = q\\
    n=2 &\implies w^2 = -1\\
    n=3 &\implies w^3 = -q
\end{align*}
Due to the cyclic nature of $\mathbb{Z}_4$ and of cosine and sine, the powers of $w$ are constant on conjugacy classes of $\mathbb{Z}_4$. 

This leads to the family $$F_4(\theta, \varphi) = \begin{bmatrix}
     1 & 1 & 1 & 1\\
     q & -1 & -q & 1\\
     -1 & 1 & -1 & 1\\
     -q & -1 & q & 1
\end{bmatrix}. $$
\end{example}

\begin{example}
Consider $\omega$ as $$\omega=\omega(\theta, \varphi) = \cos\Big(\frac{2\pi}{5}\Big) + q(\theta, \varphi) \sin\Big(\frac{2\pi}{5}\Big),$$ i.e. the fifth roots of unity where $q$ is defined above as $$q= q(\theta, \varphi)=\cos(\theta)\sin(\varphi)\imath + \sin(\theta)\sin(\varphi) \jmath + \cos(\varphi)\kappa. $$ Then we have a two parameter family of Hadamard matrices, i.e. 
$$H_5(\theta, \varphi) = \begin{bmatrix}
    1&1&1&1&1\\
    1&\omega(\theta, \varphi) &\omega(\theta, \varphi)^3&\omega(\theta, \varphi)^4&\omega(\theta, \varphi)^2\\
    1&\omega(\theta, \varphi)^2 &\omega(\theta, \varphi)&\omega(\theta, \varphi)^3&\omega(\theta, \varphi)^4\\
    1&\omega(\theta, \varphi)^4 &\omega(\theta, \varphi)^2&\omega(\theta, \varphi)&\omega(\theta, \varphi)^3\\
    1&\omega(\theta, \varphi)^3 &\omega(\theta, \varphi)^4&\omega(\theta, \varphi)^2&\omega(\theta, \varphi)^1\\
\end{bmatrix}. $$

Since $H_5\Big(0,\frac\pi2\Big)\in M_5(\complex)$ is a Hadamard matrix, we have that it is equivalent to the Fourier matrix of order 5 \cite{Haagerup5by5}, and it is circulant core.  Moreover, we have an infinite family of inequivalent quaternionic Hadamard matrices containing the complex Fourier matrix of order 5.
\end{example}

\section{Butson-type Quaternionic Hadamard Matrices}

In \cite{Butson}, Butson defined the class of generalized Hadmard matrices of order $n$ based on $k$-th roots of unity.  These matrices are now referred to as Butson-type Hadamard matrices.

\begin{definition}
The set $BH(r,n)$ is the set of Hadamard matrices of order $n$ where every entry of a element of $BH(r,n)$ is a $r^\text{th}$ root of unity.
\end{definition}
Our plan is to extend the Butson-type Hadamard matrices to the quaternionic case.

    \subsection{Butson-$q$-type Hadamard Matrices}
    
    We know from the work in \cite{quaternionpolynomial} that the solutions to the equation $x^r-1=0$ are spheres of equations generated by complex $r^\text{th}$ roots of unity. The solutions will look like $x_s = e^{\frac{2\pi \imath s}{r}}$ and any $vx_sv^{-1}$ for $v\in\quaternions\setminus\{0\}$.  This can be done for each $s\in\{0,1,\cdots, r-1\}$. Using de Moivre's identity, we have that 
    $$vx_sv^{-1} = \cos\Big(\frac{2\pi s}{r}\Big) + v\imath v^{-1} \sin\Big(\frac{2\pi}{r}\Big) = e^{\frac{2\pi (v\imath v^{-1}) s}{r}}.$$  Since $\imath^2=-1$, we have that $q^2=-1$ for $q=v\imath v^{-1}$.  In other words, we have that $r^\text{th}$ roots of unity take the form $$e^{\frac{2\pi q s}{r}}$$ for any $q\in\quaternions$ with $q^2=-1$.  For sake of simplicity, we will fix such a $q$ and define \textbf{Butson-$q$-type} Hadamard matrices.
    
    \begin{definition}
 $A$ is a quaternionic matrix of \textbf{Buston-q-type} if (for some fixed $q$ so that $q^2=-1$) $A$ is a Hadamard matrix whose entries are generated by $\omega=e^{\frac{2\pi q}{r}}$, where $r$ is a positive integer. We denote the collection of order $n$ Hadamard matrices of Buston-q-type with $rth$ roots of unity as $H(q,r,n)$.
 \end{definition}
 
 \begin{proposition}
 Let $r$, $q$, and $n$ be given so that $H(q,r,n)$ is nonempty. Then one can find a unit quaternion $u$ (with $u^{-1}=\overline{u}$) so that $uA\overline{u}$ is a complex Hadamard matrix of Butson type.
 \end{proposition}

\begin{proof}
Let $\omega=e^{\frac{2\pi q}{r}}$. Let $A$ be in $H(q,r,n)$. Then one can find a $u$ described in the proposition so that $u\omega\overline{u}$ is complex. Every other entry of $A$ is of the form $\omega^k$. Noticing that $u\omega^k\overline{u}=(u\omega\overline{u})^k$, we have the result.
\end{proof}

An immediate consequence of the preceding proposition:
\begin{corollary}
    $H(q,n,n)$ is non-empty for all $n$.  In other words, $H(q,n,n)\ne \emptyset$ if and only if $BH(n,n)\ne \emptyset$ which is true in the complex case.
\end{corollary}

This follows directly from the work shown previously in constructing families of quaternionic Hadamard matrices through the complex Fourier case.

As is known, all solutions to $x^2=-1$ have zero real part and have norm $1$, as such for any $q_1,q_2$ solutions to $x^2=-1$ there is a unital $u$ so that $uq_1\overline{u}=q_2$. A consequence of this is:

\begin{proposition}
Let $q_1$ and $q_2$ be solutions to $x^2=-1$. Then For every $A\in H(q_1,r,n)$ one can find a $u$ unital so that $uA\overline{u}\in H(q_2,r,n)$.
\end{proposition}

\begin{proof}
Select $u$ unital so that $uq_1\overline{u}=q_2$. Every entry of $A$ is generated by $\omega_1=\cos(\frac{2\pi}{r})+q_1\sin(\frac{2\pi}{r})$. Upon group conjugation by $u$, one yeilds that $uA\overline{u}$ is generated by $\omega_2=\cos(\frac{2\pi}{r})+q_2\sin(\frac{2\pi}{r})$ hence by definition $uA\overline{u}$ lies in $H(q_2,r,n)$.
\end{proof}

One wonders why in the quaternions we must declare a fixed $q$ solution of $x^2=-1$. This is because of the difficulties that arise from the following proposition.

\begin{proposition}
Let $q_1$ and $q_2$ be solutions to $x^2=-1$. If $q_1\neq q_2$ and $q_1\neq \overline{q}_2$ (i.e. $q_1\neq \pm q_2$), then there is no $u$ unital (or otherwise) so that both $uq_1\overline{u}$ and $uq_2\overline{u}$ are complex.
\end{proposition}

\begin{proof}
We prove the contrapositive. Suppose there is such a $u$ so that both $uq_1\overline{u}$ and $uq_2\overline{u}$ are complex. WOLOG we may assume $u$ is unital (see proposition 1 of section 8). Then as group conjugation preserves norm and real parts, we must have that $uq_1\overline{u}$ and $uq_2\overline{u}$ are $\pm\imath$ hence $uq_1\overline{u}=\pm uq_2\overline{u}$ which implies that $q_1=\pm q_2$.
\end{proof}

This shows that even "simple" Hadamard matrices made up of $4th$ roots of unity such as $\begin{bmatrix} 1&i\\j&k\end{bmatrix}$ can't be group conjugated into a matrix of complex entries.

\begin{proposition}
If $A$ is a Hadamard matrix containing two roots of unity derived from solutions $q_1$ and $q_2$ of $x^2=-1$, $q_1\neq \pm q_2$, then there does \textbf{not} exist a $u$ unital so that $uA\overline{u}$ is a complex Hadamard matrix of Buston type.
\end{proposition}

From our work, we get the following result:

\begin{proposition}
 $H(q,r,n)$ is non-empty if and only if $H(i,r,n)=BH(r,n)$ is non-empty.
\end{proposition}

The main observation that should be taken from this section is that any result from complex Butson type Hadamard matrices applies to Butson $q$-type Hadamard matrices.  Since if $H$ is a Butson $q$-type Hadamard matrix, we have that we may conjugate $H$ $q$ by some $0\ne u\in \quaternions$ with $uqu^{-1}=\imath$ so that $uHu^{-1}\in M_n(\complex)$.

\subsection{On general Butson type}
    
\begin{definition}
$A$ is a quaternionic matrix of Buston type if $A$ is a Hadamard matrix of some order (say n) whose entries are wholly $rth$ roots of unity for some $r$. We denote the collection of such matrices as $H(r,n)$. 
\end{definition}

Obviously, $H(q,r,n)\subseteq H(r,n)$ for all $q,r,$ and $n$. However, this subset can be strict. The matrix $\begin{bmatrix} 1&\imath\\\jmath&\kappa \end{bmatrix}$ is in $H(4,2)$ but is not in $H(q,4,2)$ for any $q\in \quaternions\setminus\reals$.

\begin{question}
    Though the matrix $\begin{bmatrix} 1&\imath\\\jmath&\kappa \end{bmatrix}$ is not a $q$-type matrix for any particular $q$, this matrix as a order 2 Hadamard matrix is equivalent to the dephased matrix $$\begin{bmatrix} 1&1\\1&-1 \end{bmatrix} = \begin{bmatrix} 1&0\\0&-\jmath \end{bmatrix} \begin{bmatrix} 1&\imath\\\jmath&\kappa \end{bmatrix} \begin{bmatrix}1&0\\ 0&-\imath \end{bmatrix}.
    $$ So the question arises, are all Hadamard matrices $q$-type or equivalent to $q$-type?  
\end{question}

This doesn't exactly ask if all Butson type matrices are Butson $q$-type, but answering this question will allow us to know if there are some quaternionic Hadamard matrices are non-$q$-type. We can answer this question in the negative. Consider the examples below.

\begin{example}
 Consider the 3-parameter generic family of order 4 quaternionic Hadamard matrices found in \cite{4x4quaternionichadamard}. 
 
 Importantly, we find that for some values of $a$ and $x$, we get that $c,d\in\quaternions$ and $c,d\not\in span\{1,\imath,\jmath\}$.
 In other words, there cannot be a single $q\in \quaternions$ such that $a, b, c, d$ can all be of the form $e^{q\theta}$ for some $\theta\in \reals$.  
 This is shown by the fact that $a\in span\{ 1,\imath,\jmath\} \implies q \in span\{1,\imath, \jmath\}$, but $b\in \complex$ which would imply that $q\in\complex$, but it's not.
 Thus, no $q$ may exist for all $a\in span\{1,\imath,\jmath\}$.  
 
 For a particular example, since $a$ and $x$ have unit length, we have that $$a = \cos(\theta)\sin(\varphi) + \sin(\theta)\sin(\varphi)\imath + \cos(\varphi)\jmath$$ and $$x = \cos(\gamma)\imath+\sin(\gamma)\jmath$$ for $\theta, \varphi, \gamma\in [0,2\pi).$  Since we want $a\ne -1$, we can choose $\theta=\varphi=\gamma=\frac\pi 4$.
 
 It follows that 
 \begin{align*}
     a &= \frac12+\frac{1}{2}\imath+\frac{1}{\sqrt{2}} \jmath\\
     b &= -\frac45-\frac35 \imath\\
     c &= \frac{1}{12} \left(2 \sqrt{2}-9\right) + \frac{1}{4} \left(-1-\sqrt{2}\right) \imath + \frac{1}{4} \left(-1-\sqrt{2}\right)\jmath -\frac{1}{12}\kappa\\
     d &= -\frac{1}{10} + \frac{3}{10}\imath + \frac{3}{10}\jmath + \frac{9}{10}\kappa.
 \end{align*}
 
 It is easily shown that $a = e^{q\frac{\pi}{3}}$ where $q = \frac{1}{\sqrt{3}}\imath + \sqrt{\frac{2}{3}}\jmath$ (note that $q^2=-1$), and $b = e^{-\imath \theta}$ where $\theta = \arccos(-\frac{4}{5})$, but since $q\ne i$, we immediately have that our Hadamard matrix is NOT $q$-type for a particular $q$.  
\end{example}

\begin{example}
For another example, we have the example at the end of the order 5 discussion.  We gave an example of a Hadamard matrix that isn't circulant core and all of the entries are fourth roots of unity.  Therefore, we have that $H(4,5)$ is non-empty.  
\end{example}

This example is interesting due to the fact that $BH(4,5)$ is empty (i.e. there are no complex Hadamard matrices of order 5 with entries that are fourth roots of unity), but in the quaternionic case, the set is non-empty.  Also note that the matrices cannot be Butson-$q$-type. We know that $\imath$ and $-1$ can be written as a $q$-type (exponential of some $q$ ($=\imath$ in the case) multiplied by some real argument), but the other two entries cannot be as they are both multiplied by $\imath$. We should also note again that since these are fourth roots of unity, the two entries that are not $\imath$ or $-1$ must be group conjugate to $\imath$ since they both have norm 1 and have zero real part. 

\section{Creating quaternionic Hadamard matrices from real and complex orthogonal matrices}
 We end this paper with ways to create quaternionic Hadamard matrices from certain real and complex Hadamard matrices. Before we begin, we note the easily shown proposition:
 \begin{proposition}
 $H$ is a Hadamard matrix over the quaternions if and only if $HH^*=nI_n$ and every entry of $H$ is of norm 1.
 \end{proposition} 
 
 Also recall from earlier in the paper that any quaternion $a+b\imath +c\jmath +d\kappa$ can be written as a specific sum of complex numbers $x+y\jmath$, $x,y\in\complex$ since $\kappa =\imath \jmath$. It follows that one can write any quaternionic Hadamard matrix $A$ as a sum $A_0+A_1\jmath$ where $A_0, A_1$ are matrices with complex entries.

\begin{definition}
Let $A=A_0+A_1\jmath$ be an $n\times n$ matrix with quaternion entries and $A_0,A_1$ with complex entries. Then the \textbf{complex adjoint} of $A$, denoted $\chi_{A}$ is the $2n\times 2n$ matrix
$$\chi_{A}=\begin{bmatrix}A_0&A_1\\-\overline{A_1}&\overline{A_0} \end{bmatrix} $$
\end{definition}

The complex adjoint has a variety of properties. An abbreviated list from section 4 of \cite{quaternionmatrixtheory} is provided below:

\begin{theorem}
Let $A$ and $B$ be $n\times n$ matrices over the quaternions. Then:
    \begin{description}
    \item{1.)} $\chi_{I_n}=I_{2n}$
    \item{2.)} $\chi_{AB}=\chi_{A}\chi_{B}$
    \item{3.)} $\chi_{A+B}=\chi_A +\chi_B$
    \item{4.)} $\chi_{A^*}=(\chi_{A})^*$
    \item{5.)} $\chi_A$ is unitary, Hermitian, or normal if and only if $A$ is unitary, Hermitian, or normal respectively.
    \end{description}
\end{theorem}

\begin{lemma}
If $\chi_{AA^*}=nI_{2n}$, then $AA^*=nI_n$
\end{lemma}
\begin{proof}
Let $AA^*=A_0+A_1\jmath$. Then $$\begin{bmatrix}nI_n&0\\0&nI_n \end{bmatrix}=\chi_{AA^*}=\begin{bmatrix}A_0&A_1\\-\overline{A_1}&\overline{A_0} \end{bmatrix} $$
The result follows.
\end{proof}

Using the lemma, one can show
\begin{theorem}
Let $A$ be a matrix over the quaternions. If $\chi_A$ is a complex Hadamard matrix of order $2n$, then $\frac{1}{\sqrt{2}}A$ is a quaternionic Hadamard matrix of order $n$. 
\end{theorem}

\begin{proof}
Using the above lemma, one can show that $AA^*=nI_n$ if and only if
$\chi_{A}(\chi_{A})^*=nI_{2n}$. This shows that the inner product of any two distinct rows or any two distinct columns are zero. It remains to show that $\frac{1}{\sqrt{2}}A$ has entries of norm 1. Since $\chi_A$ is a complex Hadamard matrix, every entry has norm 1.
By construction, each entry of $A$ will be a sum of two complex numbers with the second multiplied by $\jmath$ on the right $a+b\imath+(c+d\imath)\jmath$, where $a,b,c,d\in\reals$. If the complex numbers come from $\chi_A$, $a^2+b^2=c^2+d^2=1$ hence the norm of an entry of $A$ will be $a^2+b^2+c^2+d^2=2$.
Hence $\frac{1}{\sqrt{2}}(a+b\imath +c\jmath +d\kappa)$ will have norm 1. As this entry was arbitrary, we have that $\frac{1}{\sqrt{2}}A$ is a quaternionic Hadamard matrix.
\end{proof}
 
 Since any quaternion $q$ can be written as $q_0+q_1\imath+q_2\jmath+q_3\kappa$, one can write a quaternionic matrix $A$ as $A_0+A_1\imath+A_2\jmath+A_3\kappa$. This in turn gives rise to a real adjoint of a matrix. The following is from \cite{realadjoint}:
 
 \begin{definition}
 Let $A=A_0+A_1\imath+A_2\jmath+A_3\kappa$ be an $n\times n$ matrix. Then the \textbf{real adjoint} of $A$, denoted $\psi_A$ is the $4n\times 4n$ real matrix $$\begin{bmatrix} A_0 &A_1 &-A_2 &A_3 \\A_1 &-A_0 &-A_3 &-A_2 \\A_2 &-A_3 &A_0 &A_1 \\A_3 &A_2 &A_1 &-A_0 \end{bmatrix}.$$
 \end{definition}
 
Similar to complex adjoints, real adjoints have similar properties, one from \cite{realadjoint} is listed here.
 \begin{proposition}
 $\psi_A$ is orthogonal if and only if $A$ is unitary.
 \end{proposition}
 
 We now create a similar result to the previous proposition:
 
 \begin{proposition}
 If $\psi_A$ is a real Hadamard matrix of order $4n$, then $\frac{1}{2}A$ is a quaternionic Hadamard matrix of order $n$.
 \end{proposition}
 
  Note that these propositions work only if of one has a specific kinds of Hadamard matrices. If complex, one needs a Hadamard matrix of the form $\begin{bmatrix}A_0&A_1\\-\overline{A_1}&\overline{A_0} \end{bmatrix}$. If real, one needs a Hadamard matrix of the form $$\begin{bmatrix} A_0 &A_1 &-A_2 &A_3 \\A_1 &-A_0 &-A_3 &-A_2 \\A_2 &-A_3 &A_0 &A_1 \\A_3 &A_2 &A_1 &-A_0 \end{bmatrix}.$$
  One wonders how one can find such complex or real Hadamard matrices.
 
 \begin{definition}
 A complex (or real) Hadamard matrix $A$ is \textbf{quaternionically compliant} if it is equivalent to the forms discussed above.
 \end{definition}

\begin{question}
Given a positive integer $n$, can one find quaternionically compliant complex Hadamard matrices or order $2n$ (or real quaternionically compliant matrices of order $4n$)?
\end{question}





\begin{thebibliography}{1}

\bibitem{Brenner} J. L. Brenner,\emph{Matrices of Quaternions}, Pacific Journal of Mathematics, vol. 3, 1951.

\bibitem{Butson} A.T. Butson, \emph{Generalized Hadamard Matrices}, Proceedings of the American Mathematical Society, vol. 13 (6), Dec. 1962.

\bibitem{4x4quaternionichadamard} Oleg Chterental and Dragomir Dokovic, \emph{On OrthoStochastic, Unistochastic and Qustochastic Matrices}, Linear Algebra and its Applications, vol. 428, 2008.

\bibitem{4x4complexhadamard} R. Craigen,
\emph{Equivalence Classes of Inverse Orthogonal and Unit Hadamard  Matrices}, Bulletin of the Australian Mathematical Society, vol. 4, 1991.


\bibitem{Haagerup5by5} Uffe {Haagerup}, \emph{Orthogonal maximal abelian $*$-subalgebras of the $n\times n$ matrices and cyclic $n$-roots}, Operator Algebras and Quantum Field Theory, 1997.

\bibitem{quaternionpolynomial} Bahman Kalantari,
\emph{Algorithms for quaternion polynomial root-finding}, Journal of Complexity, vol. 29, 2013.


\bibitem{complexadjoint} H. C. Lee, \emph{Eigenvalues and Cannonical Forms of Matrices with Quaternion Coefficients}, Proceedings of the Royal Irish Academy Section A, vol. 52, 1949.

\bibitem{Ni1} Remus Nicoara, \emph{A finiteness result for commuting squares of matrix algebras}, J. of Operator Theory, vol. 55, 2006.

\bibitem{NiBe} Remus Nicoara and Kyle Beauchamp, \emph{Maximal Abelian *-algebras of the 6x6 matrices}, Journal of Linear Algebra and its Applications, vol. 428, pages 1833-1853, 2006.

\bibitem{NiWh} Remus Nicoara and Joseph White,
\emph{The defect of a group-type commuting square}, Revue Roumaine Math., vol. 2, 2014.


\bibitem{NiWor} Remus Nicoara and Chase Worley, \emph{A Finiteness Result for circulant core complex Hadamard matrices}, Linear Algebra and its Applications, vol. 571, pages 143-153, 2019.

\bibitem{Pet} M. Petrescu,\emph{Existence of continuous families of complex Hadamard matrices of certain prime dimensions and related results}, PhD Thesis, Univ. of California Los Angeles, 1997.

\bibitem{realadjoint} Caiqin Song et al,
\emph{A Real Representation Method for Solving Yakubovich-$j$-Conjugate Quaternion Matrix Equation}, Abstract and Applied Analysis, 2013.

\bibitem{goodthesis} Ferenc Szollosi, \emph{Construction, classification and parameterization of complex Hadamard matrices}, https://sierra.ceu.edu/record=b1163883

\bibitem{quaternionalgebras} John Voight,
\emph{Quaternion Algebras}, Springer, 2021.

\bibitem{WorleyThesis} C. Worley,
\emph{Construction and Classification Results for Commuting Squares of Finite Dimensional *-algebras}, PhD Thesis, Univ. of Tennessee, 2017, https://trace.tennessee.edu/utk\_graddiss/4716.


\bibitem{quaternionmatrixtheory} Fuzhen Zhang, \emph{Quaternions and Matrices of Quaternions}, Linear Algebra and Its Applications,
vol. 251, 1997.
\end{thebibliography}
\end{document}